\documentclass[letterpaper]{article} 
\usepackage{aaai23}  
\usepackage{times}  
\usepackage{helvet}  
\usepackage{courier}  
\usepackage[hyphens]{url}  
\usepackage{graphicx} 
\usepackage{amsmath}
\usepackage{amssymb}
\usepackage{amsthm}
\usepackage{booktabs}
\usepackage{subcaption}
\usepackage{natbib}  
\usepackage{caption} 
\usepackage{algorithm}
\usepackage{algorithmic}
\newtheorem{theorem}{Theorem}
\usepackage{makecell}

\usepackage[disable]{todonotes}

\usepackage{newfloat}
\usepackage{listings}
\DeclareCaptionStyle{ruled}{labelfont=normalfont,labelsep=colon,strut=off} 
\floatstyle{ruled}
\newfloat{listing}{tb}{lst}{}
\floatname{listing}{Listing}
\newcommand{\problem}{\textsc{WalkOpt}}

\title{Walkability Optimization: Formulations, Algorithms, and a Case Study of Toronto}
\author{
    Weimin Huang,
    Elias B. Khalil
}
\affiliations{
     Department of Mechanical \& Industrial Engineering, University of Toronto\\
     Scale AI Research Chair in Data-Driven Algorithms for Modern Supply Chains\\
    \url{cheryl.huang@mail.utoronto.ca}, \url{khalil@mie.utoronto.ca}
}

\usepackage{bibentry}

\begin{document}

\maketitle

\begin{abstract}
The concept of walkable urban development has gained increased attention due to its public health, economic, and environmental sustainability benefits. Unfortunately, land zoning and historic under-investment have resulted in spatial inequality in walkability and social inequality among residents. We tackle the problem of Walkability Optimization through the lens of combinatorial optimization. The task is to select locations in which additional amenities (e.g., grocery stores, schools, restaurants) can be allocated to improve resident access via walking while taking into account existing amenities and providing multiple options (e.g., for restaurants). To this end, we derive Mixed-Integer Linear Programming (MILP) and Constraint Programming (CP) models. Moreover, we show that the problem's objective function is submodular in special cases, which motivates an efficient greedy heuristic. We conduct a case study on 31 underserved neighborhoods in the City of Toronto, Canada. MILP finds the best solutions in most scenarios but does not scale well with network size. The greedy algorithm scales well and finds near-optimal solutions. Our empirical evaluation shows that neighbourhoods with low walkability have a great potential for transformation into pedestrian-friendly neighbourhoods by strategically placing new amenities. Allocating 3 additional grocery stores, schools, and restaurants can improve the ``WalkScore" by more than 50 points (on a scale of 100) for 4 neighbourhoods and reduce the walking distances to amenities for 75\% of all residential locations to 10 minutes for all amenity types. Our code and paper appendix are available at \url{https://github.com/khalil-research/walkability}.

\end{abstract}

\section{Introduction}

The concept of walkability in urban planning has gained increased attention as research has shown that good walkability contributes to physical health, economic development, and environmental sustainability \cite{ZapataDiomedi2019PhysicalAH}. Highly walkable neighbourhoods allow for residents to easily access amenities such as retail and food in the vicinity, giving rise to the concept of the ``15-minute city" \cite{whittle2020welcome}. However, the zoning regulations introduced in the early 20th century that separate industrial, commercial, and residential areas have prevented walkable development and contributed to automobile-reliant communities \cite{Levine2005ZonedOR,Fischel2003AnEH} -- residents have to travel outside their local communities in order to meet their daily needs. Moreover, historic disinvestment in segregated neighbourhoods including low-income groups and racial minorities has led to spatial inequality in walkability. Evidence shows that disadvantaged groups live in neighbourhoods with less accessible physical infrastructure and services \cite{Massey1990AmericanAS}. With the COVID-19 pandemic reshaping the relationship between cities and the quality of life, addressing inequities and improving access to services and amenities such as healthcare and green spaces for vulnerable groups has also accelerated the need for walkable development~\cite{MOURATIDIS2021105772}.

Improving walkability is also a powerful way to reduce greenhouse gas emissions (GHG) and tackle climate change. Dense and walkable neighbourhoods encourage active transport (walking, cycling), thus reducing automobile dependence \cite{MCINTOSH201495,BRAND2021102224}. Research shows that technological measures (e.g., increasing the use of electric vehicles) alone will not be sufficient in reducing GHG, whereas a shift to a more sustainable mode of transportation can result in a quicker and more significant reduction of emissions from vehicles, particularly in urban areas \cite{Creutzig2018TowardsDS,Neves2019AssessingTP}. Shifting from motorized transport to active transport is considered one of the most promising ways to reduce GHG. A study shows that active walking as a lifestyle change of residents can significantly reduce emissions related to private vehicles, even in European cities that are already highly walkable \cite{BRAND2021102224}. Building inclusive, safe, resilient, and sustainable cities has been highlighted in the Sustainable Development Goal 11 of the United Nations. 

Towards this goal, researchers have been interested in improving walkability in cities by converting certain underused spaces into easily accessible amenities. For instance, some urban planning research relies on simulation and inspection, such as converting high-density regions into amenities and straightening busy routes \cite{qualitative}.
On the other hand, one common optimization-based approach is to use a Genetic Algorithm (GA) that encodes the potential allocation locations (e.g., empty lots where a grocery store can be built) as fixed-sized vectors and iteratively generate child solutions from a population of candidate solutions~\cite{gen1,gen2,gen3,gen4,Indraprastha2019InformedWC}. However, genetic algorithms lead to sub-optimal solutions \cite{dis3} and lack convergence guarantees. Moreover, GAs are not flexible in handling constraints \cite{dis1}, making them less applicable in the context of modern city planning, given the existing city layout and property ownership.
Despite this interest in walkability optimization, there have not been robust and efficient methods for this problem.

\textbf{Contributions:} To our knowledge, the problem of \textit{Walkability Optimization} has not yet been examined from an algorithmic perspective, particularly one that considers the problem's realistic aspects. In this paper, we contribute to this question along multiple axes:
\begin{enumerate}
    \item We formulate \textit{Walkability Optimization} as the combinatorial optimization problem of selecting the locations of new amenities to maximally improve residents' access to basic necessities (Section \ref{sec_formulatoin}). Our formulation extends the standard facility location problem and models residents' behaviour realistically -- we consider multiple facility types, multiple potential choices for the same type, and an objective function with respect to the travel distances that represent the proximity to residents. Also, compared to existing work on Walkability Optimization, our formulation takes into account existing amenities rather than designing a layout from scratch and is flexible in capturing additional constraints on allocation.
    \item We analyze the complexity of the problem, showing that (a) its decision version is NP-Complete in general and that (b) the objective function is submodular in special cases (Section \ref{sec_theory}). This latter property motivates an efficient greedy algorithm presented in Section \ref{sec_models}.
    \item We derive Mixed-Integer Linear Programming (MILP) and Constraint Programming (CP) models (Section \ref{sec_models}).
    \item  We perform a case study of 31 underserved neighbourhoods identified by the City of Toronto, Canada (Section \ref{sec_case_study}). In most neighborhoods, significant reductions in walking distances can be obtained by optimizing the placement of a handful of new amenities. MILP outperforms CP and the simple greedy algorithm achieves a good tradeoff between running time and solution quality.  
\end{enumerate}

\section{Problem Formulation}\label{sec_formulatoin}

At a high level, the Walkability Optimization problem (\problem) is defined as follows. Given a set $N$ of residential locations, a set $M$ of candidate allocation locations, and a set $L$ of existing amenities, we seek a set of locations where new amenities of different types $A$ are allocated so that the average ``Walkability Score" $f_i$ for residential locations $i \in N$ is maximized. $f_i$ is a function of walking distances and can be interpreted as the proximity of $A$ to residents. The maximum number of instances to be allocated for each amenity type $a \in A$ is denoted as $k_a$. The set of locations $M$, $N$, and $L$ correspond to nodes on a network where edges represent the walkable paths in the neighbourhoods. The walking distances between these locations are the shortest-path distances, which we denote as $d_{ij}$ for $i \in N$,$j \in M \cup L$. Note that our formulation improves walkability by introducing amenities instead of other potential approaches such as improving network connectivity by adding edges. The latter is a challenging intervention since road networks are resistant to change compared to the speed of city expansion \cite{gen2}.


\subsection{Building Blocks}

\subsubsection{Walkability Score} 
To quantify walkability, we adopt the WalkScore methodology \cite{walk_score}, used in prior quantitative analyses \cite{walk_measure_1,Verbas2015StretchingRS}, that assigns a score $f_i$ for each residential location $i$ based on the weighted walking distances, denoted as $l_i$, from $i$ to different amenity types. The WalkScore assigns a near full score (100) for distances within 400m with scores decreasing with respect to $l_i$ after 400m. Distances above 2400m (about a 30-minute walk) are not rewarded any points. The WalkScore $f_i$ is originally an exponential function that models this non-linear decay with respect to $l_i$; see Fig.~\ref{pwl_score} (a). For computational purposes, nonlinear functions can be approximated with Piecewise Linear functions (PWLs) \cite{Ngueveu2019PiecewiseLB}: $f_i$ is represented as a PWL that is parameterized by breakpoints $\bar{t}$ (Fig. \ref{pwl_score} (b)). The parameters $\bar{t}$ are shown in Appendix \ref{appendix_pwl}.
\begin{figure}[]
  \begin{subfigure}[b]{0.48\columnwidth}
         \includegraphics[width=.9\linewidth]{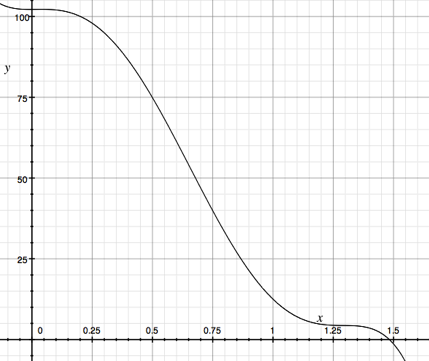}
         \caption{Original (Exponential)}
     \end{subfigure}
     \hfill
     \begin{subfigure}[b]{0.50\columnwidth}
         \includegraphics[width=1\linewidth]{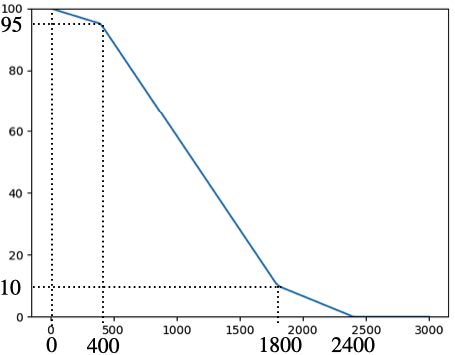}
         \caption{Approximation (PWL)}
     \end{subfigure}
  \caption{Distance (miles/meters) v.s. WalkScore.}
  \label{pwl_score}
\end{figure}

\subsubsection{Weighted Walking Distances} The weighted walking distance $l_i$ is a linear combination of distances to multiple amenity types $a \in A$. Each type is given a weight $w_a$ based on its level of necessity, e.g., grocery stores are given the highest weight as they are the most frequent walking destination \cite{grocery_weight}. For most amenity types, we assume that residents at location $i$ will walk for $D_{i,a}$ meters to the nearest instance of the amenity $a$; we denote amenities of this type with $A^{plain}$. For amenity types for which variety and options are important  (e.g., restaurants, coffee shops), we account for the resident's ``depth of choice" by considering a weighted combination of the distances to the top-$r$ nearest instances (e.g., $r=10$ for restaurants): the $p^{th}$ nearest instance (e.g., of a restaurant) has distance $D_{i,a}^p$ and weight $w_{a}^p$. We denote such amenities with $A^{depth}$. We have $A^{plain} \cup A^{depth} = A$. Finally, $l_i$ is expressed as
\begin{equation}\label{walk_formula}
l_i = \sum_{a \in A^{plain}} w_a D_{i,a}  + \sum_{a \in A^{depth}} \sum_{p \in P_a} w_{a}^p D_{i,a}^p, \forall i \in N,
\end{equation}
where $P_a$ is the set of top-$r$ nearest instances of amenity type $a\in A^{depth}$. Note that not all $r$ options in $P_a$ may be available in the network, as the sum of the number of existing and allocated instances may be smaller than $r$. For example, there may be 2 existing restaurants and a budget of 3 additional restaurants to be allocated, but the number of choices being considered is $10>3+2$. The set of \textit{available} choices of type $a$ is denoted as $P_a^{Y}$, where $|P_a^{Y}| = min(k_a + |L_a|, |P_a|)$. The set of choices that we consider but are not available is denoted as $P_a^{N} = P_a \setminus P_a^{Y}$. Choices $p \in P_a^{N}$ have distances $D_{i,a}^p=D^{\infty}$, where $D^{\infty}=2400$m (i.e., WalkScore is zero), so that a neighbourhood with no amenity has $f_i=0$. 

\subsubsection{Objective Function}
We maximize the average WalkScore across all residential locations 
\begin{equation}\label{walkscore_obj}
    F=\frac{1}{|N|} \sum_{i\in N} f_i.
\end{equation}
$f_i$ is a PWL function of $l_i$ parameterized by breakpoints $\bar{t}$:
\begin{equation}\label{pwl_of_dist}
    f_{i} = \text{PiecewiseLinear}(l_i,\bar{t}), \forall i \in N.
\end{equation}

\subsubsection{Existing and Candidate Amenity Locations}
We introduce new amenities to built neighbourhoods with consideration of existing amenities, instead of allocating from scratch. The set of locations with existing instances of amenity type $a$ is denoted as $L_a$. For candidate allocation locations $M$, we set a capacity on the number of amenity instances that each candidate location $j$ can accommodate (denoted as $c_j$, $j \in M$) to represent potential physical constraints. Given the context of built cities, $M$ can be underused spaces such as parking lots, as freeing up parking spaces and converting them into amenities has been recognized as an urban renewal model that reduces air pollution and improves the quality of life \cite{parking_1,parking_2}. However, our framework can accommodate various types of locations.

\section{Theoretical Properties}\label{sec_theory}
\subsection{Computational Complexity}
We prove that \problem~is computationally hard even with a single amenity type and without depth of choice. This is achieved through a reduction from the widely studied and NP-Complete $k$-median problem to the decision version of \problem.
The proof is deferred to Appendix \ref{appendix_np_compleness}.

\subsection{Submodularity}
Seeking a polynomial-time approximation algorithm for \textsc{WalkOpt}, we analyzed its objective function for submodular structure. Formally, a set function $F\colon 2^V \to\mathbb{R}$ is submodular if it satisfies the \textit{diminishing returns} property: for every $S \subseteq T \subseteq V$ and $e \in V \setminus T$ it holds that $\Delta_F{(e|S)} \geq \Delta_F{(e|T)}$, where $\Delta_F{(e|S)}:=F(S \cup \{e\}) - F(S)$ is the \textit{discrete derivative} of $F$ at $S$ w.r.t. $e$.

We show that the objective (\ref{walkscore_obj}) is indeed submodular when depth of choice is not considered. This motivates the use of a greedy algorithm as a $ (1 - \frac{1}{e})$-approximation~\cite{Nemhauser1978AnAO} when there is a single amenity type, and as a heuristic (with no guarantees) otherwise. We also show that submodularity does not hold when considering depth of choice.

\subsubsection{Submodularity in the SingleChoice Case}
\begin{theorem}
Objective~\eqref{walkscore_obj} is submodular when $A^{depth}=\emptyset$. 
\end{theorem}
\begin{proof}
First, we represent a solution as a set of actions $S$ where each element $e=(a,m) \in S$ consists in allocating an instance of amenity type $a$ to a candidate location $m$. Note that it is feasible to have more than one instance of type $a$ allocated to the same location, which can introduce identical $(a,m)$ pairs. Since sets cannot contain duplicates, we construct an equivalent set $\bar{M}$ by duplicating each node $j \in M$ for $c_j$ times. Then, the ground set is $V = \{(a,m): a \in A^{plain},m \in \bar{M}\}$ and objective~\eqref{walkscore_obj} is a set function, $F\colon 2^V \to\mathbb{R}$.
Let $S$ and $T$ be solution sets such that $S \subseteq T \subseteq V$, and let $e=(a',m') \in V \setminus T$. We show that $\Delta_F{(e|S)} \geq \Delta_F{(e|T)}$.

We denote the weighted walking distances at $i$ under solution set $S$ as $l_i^{S}$ and express $\Delta_F{(e|S)}$ in terms of $l_i^{S}$:
\begin{equation}\label{submodulardef}
    \Delta_F{(e|S)} = \frac{1}{|N|} \sum_{i \in N} (f(l_i^{S \cup \{e\}}) -  f(l_i^{S})).
\end{equation}
In this SingleChoice case, we have $A=A^{plain}$ and we consider the nearest choice for each $a$. Then, the distances are:
\begin{equation}\label{l_large_set}
    l_i^{S \cup \{e\}} =  w_{a'} \min_{j \in J_{a'}^S \cup L_{a'} \cup \{m'\} } d_{ij} + u,
\end{equation}
\begin{equation}\label{l_small_set}
    l_i^{S} =  w_{a'} \min_{j \in J_{a'}^S \cup L_{a'} } d_{ij} + u.
\end{equation}
$u$ is the weighted distance to types $a \in A \setminus\{a'\}$, which is not affected by $e=(a',m')$: 
\begin{equation*}
    u=\sum_{a \in A \setminus\{a'\}} w_a \min_{j \in J_{a}^S \cup L_{a}} d_{ij}.
\end{equation*}
$J_{a'}^S$ is the set of locations allocated for type $a'$ under solution set $S$. $L_{a'}$ is the set of existing instances for $a'$. For simplicity, we denote the minimum distance to $a'$ under $S$ as $D_{ia'}^S$:
\begin{equation*}
    D_{ia'}^S=\min_{j \in J_{a'}^S \cup L_{a'} } d_{ij}.
\end{equation*}
Then, from Eqn. (\ref{l_small_set}) we have $l_i^{S}=w_{a'} D_{ia'}^S + u$. From Eqn. (\ref{l_large_set}), when the new location $m'$ does not reduce the distance to $a'$ under $S$ (i.e., $d_{im'} \geq D_{ia'}^S$), we have $l_i^{S \cup \{e\}} = l_i^{S}$ and thus $f(l_i^{S \cup \{e\}}) = f(l_i^{S})$. When $m'$ results in a new minimum distance (i.e., $d_{im'}<D_{ia'}^S$), we have $l_i^{S \cup \{e\}}=w_{a'} d_{im'}+u$. Therefore, $\Delta_F{(e|S)}$ in Eqn. (\ref{submodulardef}) is:
\begin{equation*}\label{delta_at_s}
\begin{split}
\Delta_F{(e|S)}={} &
      \frac{1}{|N|} \sum_{d_{im'} \geq D_{ia'}^S} (f(l_i^{S \cup \{e\}}) -  f(l_i^{S})) +\\
    & \frac{1}{|N|} \sum_{d_{im'} < D_{ia'}^S} (f(l_i^{S \cup \{e\}}) -  f(l_i^{S})) \\
\end{split}
\end{equation*}
which simplifies to
\begin{equation*}
    \Delta_F{(e|S)}= \frac{1}{|N|} \sum_{\text{\scalebox{0.7}{$d_{im'} < D_{ia'}^S$}}} (f(w_{a'} d_{im'}+ u) -  f(w_{a'} D_{ia'}^S + u)).
\end{equation*}
Similarly, for solution set $T$, we have 
\begin{equation*}\label{delta_at_t}
\Delta_F{(e|T)} 
    = \frac{1}{|N|} \sum_{\text{\scalebox{0.7}{$d_{im'} < D_{ia'}^T$}}} (f(w_{a'} d_{im'}+u) -  f(w_{a'} D_{ia'}^T + u)),
\end{equation*}
where $J_{a'}^T$ is the set of locations allocated for type $a'$ under $T$. Since $S \subseteq T$, we have $J_{a'}^S \subseteq J_{a'}^T$, which leads to $D_{ia'}^T \leq D_{ia'}^S$. $\Delta_F{(e|S)}$ can then be further grouped into two cases:
\begin{equation*}\label{delta_at_s_rearrange}
\begin{split}
&\Delta_F{(e|S)} =\\& \frac{1}{|N|}\Bigg[ \sum_{d_{im'} < D_{ia'}^T} (f(w_{a'} d_{im'}+u) -  f(w_{a'} D_{ia'}^S + u))+\\
    &  \sum_{D_{ia'}^T \leq d_{im'} <D_{ia'}^S} (f(w_{a'} d_{im'}+u) -  f(w_{a'} D_{ia'}^S + u))\Bigg].
\end{split}
\end{equation*}
We compute and re-arrange $\Delta=\Delta_F{(e|S)} -\Delta_F{(e|T)} $:
\begin{equation*}
\begin{split}
    &\Delta=\\&\frac{1}{|N|}\Bigg[  \sum_{d_{im'} < D_{ia'}^T} (f(w_{a'} D_{ia'}^T + u)) -  f(w_{a'} D_{ia'}^S + u)) + \\
    &  \sum_{D_{ia'}^T \leq d_{im'} <D_{ia'}^S} (f(w_{a'} d_{im'}+u) -  f(w_{a'} D_{ia'}^S + u))\Bigg].
\end{split}
\end{equation*}
We know that $(w_{a'} D_{ia'}^T + u) \leq (w_{a'} D_{ia'}^S + u)$ from $S \subseteq T$ and that $(w_{a'} d_{im'}+u) \leq (w_{a'} D_{ia'}^S + u)$ by definition of the second summation. Since WalkScore $f()$ is monotonically decreasing, we have $\Delta \geq 0$. We've thus proved that $\Delta_F{(e|S)} \geq \Delta_F{(e|T)}$, as desired.
\end{proof}
\subsubsection{No Submodularity with Depth of Choice} We show that the objective function is not submodular by providing a counter-example in Appendix \ref{appendix_counter}.

\section{Models and Algorithms}\label{sec_models}
\subsection{Mixed-Integer Linear Programming (MILP)}
\subsubsection{Variables}
Our MILP model has four sets of variables. First, for allocation, integer variable $y_{ja}$ indicates the number of amenities of type $a$ allocated to location $j$. Second, binary variables are used to indicate the assignment of amenities to residents. For amenity types where only the nearest instance is considered, $x_{ija}=1$ indicates that residents at location $i$ visit location $j$ for type $a$. For types where depth of choice is considered, a fourth index is used and $x_{ija}^p=1$ indicates that residents at location $i$ visit location $j$ for the $p^{th}$ nearest instance of type $a$. There are also two sets of continuous variables: $l_i$ and $f_i$ represent the weighted distance and WalkScore for $i$, respectively. The number of discrete decision variables in the model is $O(|M||N|(|A^{plain}| + |A^{depth}|h))$, where $h=\max_{a \in A^{depth}}|P_a^{Y}|$.

\subsubsection{Constraints}
First, we enforce the requirements on the maximum number of amenities to be allocated and the capacity of each candidate allocation location:
\begin{equation*}
        \sum_{j \in M} y_{ja} \leq k_a , \forall a \in A,
\end{equation*}
\begin{equation*}
    \sum_{a \in A} y_{ja} \leq c_j , \forall j \in M.
\end{equation*}

Second, we describe the assignment of amenities to residents. We ensure that each resident is assigned to one instance for types $a \in A^{plain}$ and to one instance for each available choice for types $a \in A^{depth}$:
\begin{equation*}
    \sum_{j \in M \cup L_a} x_{ija} = 1, \forall i \in N , a \in A^{plain},
\end{equation*}
\begin{equation*}
\sum_{j \in M \cup L_a} x_{ija}^p = 1, \forall p \in P_a^{Y}, \forall i \in N, \forall a \in A^{depth}.
\end{equation*}

Note that for $a \in A^{depth}$, each choice $p \in P_a^{Y}$ should be a different instance of $a$. When the choice is assigned to an existing instance of an amenity type ($j \in L_a$), we ensure that the instance appears only once among all choices for each resident. When the choice corresponds to candidate locations ($j \in M$), we ensure that the number of choices provided at $j$ does not exceed the number of instances allocated to $j$:
\begin{equation*}\label{choice_exist}
    \sum_{p \in P_a^{Y}} x_{ija}^p \leq 1, \forall j \in L_a, \forall i \in N,\forall a \in A^{depth},
\end{equation*}
\begin{equation*}\label{choice_allocate}
    \sum_{p \in P_a^{Y}} x_{ija}^p \leq y_{ja}, \forall j \in M, \forall i \in N,\forall a \in A^{depth}.
\end{equation*}

Additionally, we ensure that all the amenities are allocated before they are assigned:
\begin{equation*}
        x_{ija} \leq y_{ja}, \forall i \in N, \forall j \in M \cup L_a, \forall a \in A^{plain},
\end{equation*}
\begin{equation*}
    x_{ija}^p \leq y_{ja}, \forall i \in N, \forall j \in M \cup L_a,\forall p \in P_a^{Y},\forall a \in A^{depth}.
\end{equation*}
Finally, we describe the weighted walking distances based on Eqn. (\ref{walk_formula}) and the PWL WalkScore:
\begin{equation*}
\begin{split}
l_i & =  \sum_{a \in A^{depth}} \Big(\sum_{p \in P_a^{Y}} w_{a}^p \sum_{j \in M \cup L_a} x_{ija}^pd_{ij} + \sum_{p \in P_a^{N}} w_{a}^p D^{\infty}\Big)\\
&+\sum_{a \in A^{plain}} w_a \sum_{j \in M \cup L_a} x_{ija}d_{ij},\quad \forall i \in N.
\end{split}
\end{equation*}
For the PWL in Eqn. (\ref{pwl_of_dist}), commercial MILP solvers provide the functionality for linearizing PWL functions \cite{grp_wpl}.

\subsubsection{Objective} The objective is to maximize $F$ in Eqn. (\ref{walkscore_obj}).

\subsection{Constraint Programming Model (CP)}
We also provide a CP model for \problem. CP allows for an index-based formulation in which decision variables indicate the index of the location of each instance of $a \in A$. This significantly reduces the number of discrete variables compared to the binary formulation in MILP, particularly in the case without depth of choice. Specifically, the number of discrete decision variables is $O(k|A|+h|N||A^{depth}|)$ where $k=\max_{a \in A}k_a$ and $h=\max_{a \in A^{depth}}|P_a^{Y}|$. The model is deferred to Appendix \ref{appendix_cp}.



\subsection{Greedy Algorithm (Greedy)}

Motivated by the submodularity in the SingleChoice case, we use a greedy algorithm that iteratively selects the (amenity type, location) pair that maximizes the increase in the objective. We also use Greedy as a heuristic when considering depth of choice. Greedy runs in time $O(k|M||N|(|A^{plain}|+h|A^{depth}|))$ and
is shown in Algorithm \ref{alg:greedy}. The solution set $S$ contains $(a,j)$ pairs that represent the action of allocating an instance of type $a$ to location $j$. We denoted the weighted walking distance of $i$ under solution set $S$ as $l_i^{S}$. The WalkScore function is denoted as $f()$. 

\begin{algorithm}[t]
\caption{Greedy Algorithm}
\label{alg:greedy}
\begin{algorithmic}[1] 
\STATE $S \gets \emptyset$.
\STATE $n_a \gets 0$ for all $a\in A$.  \COMMENT{$n_a$: Number of allocated instances for type $a$}
\STATE $c[j] \gets c_j$ for all $j\in M$.
\WHILE{$\exists n_a < k_a$ \AND $\max_{j \in M}(c[j])>0$}
\STATE $(a,j) \gets {\mathrm{argmax}}_{\substack{a \in A\\ n_a < k_a,j \in M\\c[j]>0}}\, \sum_{i \in N}f(l_i^{S \cup \{(a,j)\}})$
\STATE $S \gets S \cup {(a,j)}$
\STATE $n_a \gets n_a + 1 $ 
\STATE $c[j] \gets c[j]-1$

\ENDWHILE
\STATE \textbf{return} $S$
\end{algorithmic}
\end{algorithm}

\section{Case Study}\label{sec_case_study}

We perform a case study for 31 underserved neighbourhoods in Toronto, Canada. The City of Toronto has identified 31 of its 140 neighbourhoods as \textit{Neighbourhood Improvement Areas} (NIAs) that are facing the most inequitable outcomes under the Toronto Strong Neighbourhood Strategy (TSNS) 2020 \cite{NIA}. TSNS aims to provide equitable social, economic, and cultural opportunities for all residents by partnering with agencies to invest in services and facilities in neighbourhoods that face historic under-investment, and the NIAs capture areas of the city with a significant concentration of disadvantaged and equity-seeking groups, particularly visible minorities. Low walkability and limited access to amenities in the physical surroundings are important criteria in the selection of the NIAs \cite{TSNS2020}. 

\subsection{Data}
\subsubsection{Neighbourhood Improvement Areas (NIAs)}
We create instances of \problem~ from each NIA. The geographical boundary of the NIAs is publicly available from The City of Toronto’s Open Data Portal \cite{data_NIA}. 

\subsubsection{Pedestrian Network} 
To obtain the network of walkable paths, we use a publicly available \textit{Pedestrian Network} (PedNet) of Toronto which includes various pedestrian assets such as sidewalks, crosswalks, and pedestrian-controlled crossings that are topologically consistent \cite{data_pednet_git}. We precompute the shortest-path pairs based on PedNet for each NIA\todo{$d_{ij}$ using NetworkX \cite{hagberg2008exploring}}. An alternative to PedNet is OpenStreetMap which provides walking routes data worldwide, which may enable case studies in other cities in the future. We use PedNet in our case study since it undertook quality assurance and has been used in a walking time assessment report by the Transportation Services and Information and Technology Division \cite{report2}. 


\subsubsection{Residence and Candidate Locations} 
The locations of residential areas and potential allocation candidates are obtained from OpenStreetMap \cite{OpenStreetMap} and mapped to the nearest nodes in the PedNet. Residential nodes $N$ are the set of nodes that map to at least one residential address. As mentioned in Section \ref{sec_formulatoin}, our candidate allocation nodes $M$ are parking lots. For this case study, the capacity of each candidate node $c_j$ is the number of parking lots mapped to the node $j$.


\subsubsection{Amenity Weights}
In this case study, we consider 3 types of amenities: grocery stores, restaurants, and schools, for which the locations of existing instances are also obtained from OpenStreetMap. These 3 types are chosen because they are the major categories that the WalkScore methodology considers \cite{walk_score} and the data provided by OpenStreetMap is relatively rich for these 3 types based on our visual inspection. The weights $w_a$ (for different types) and $w_a^p$ (for different choices) are obtained from the WalkScore methodology documentation \cite{walk_score}, and the values are listed in Appendix \ref{appendix_amenity_weights}.

\subsection{Computational Setup}

We perform experiments under two scenarios. In the first  (MultiChoice), the distance to the nearest instance is considered for groceries and schools, while the distances to the top 10 nearest choices are considered for restaurants. In the second (SingleChoice), the distance to the nearest instance is considered for all 3 amenity types. MultiChoice is consistent with the original WalkScore methodology \cite{walk_score}, while the SingleChoice scenario helps assess the effect of depth of choice on walkability and solving difficulty.

\subsubsection{Instances}
Throughout the case study, we use the same upper bound for the 3 amenity types considered: $k_a=k,\forall a \in A$. We create 9 instances for each NIA with $k \in \{1,2,...,9\}$. The 31 NIAs are split into 4 groups according to the size ($|M|+|N|$). The number of NIAs (\# NIA) and the number of instances (\# Inst.) for each NIA group are shown in Table \ref{table_instances}.

\begin{table}[h]
    \centering
\resizebox{\columnwidth}{!}{
\begin{tabular}{lllllll}
\toprule
   \thead{Group} &  \thead{Network Size\\ $|M|+|N|$} &  \thead{\# NIA} & \thead{\# Inst.} & Groceries & Restaurants & Schools\\
\midrule
1 &   [0,200) &    11 &   99 &          1.45 &    5.73 &       3.18 \\
2 & [200,400) &    13 &   117 &          2.00 &   11.85 &       5.46 \\
3 & [400,600) &     4 &   36 &          4.50 &   14.50 &       9.25 \\
4 & [600,1200) &    3 &    27 &          6.67 &   19.00 &      14.33 \\
\bottomrule
\end{tabular}
}

\caption{Instance statistics. The last 3 columns are averages.}\label{table_instances}

\end{table}

\subsubsection{Setup}
All methods are implemented in Python. MILP and CP models are solved with Gurobi and CP Optimizer, respectively, and 8 threads.  Experiments were run on Intel E5-2683 v4 Broadwell at 2.1GHz CPUs and a memory limit of 32GB. Each solving run is limited to 5 hours. We use a relatively large time limit since a few hours of computation is tiny compared to the service time (in years) of an amenity. 

\section{Results}
\subsection{Comparison of Solution Methods}

\subsubsection{Solving time}
\begin{figure}[t]
  \begin{subfigure}[b]{0.565\columnwidth}         \includegraphics[width=1\linewidth]{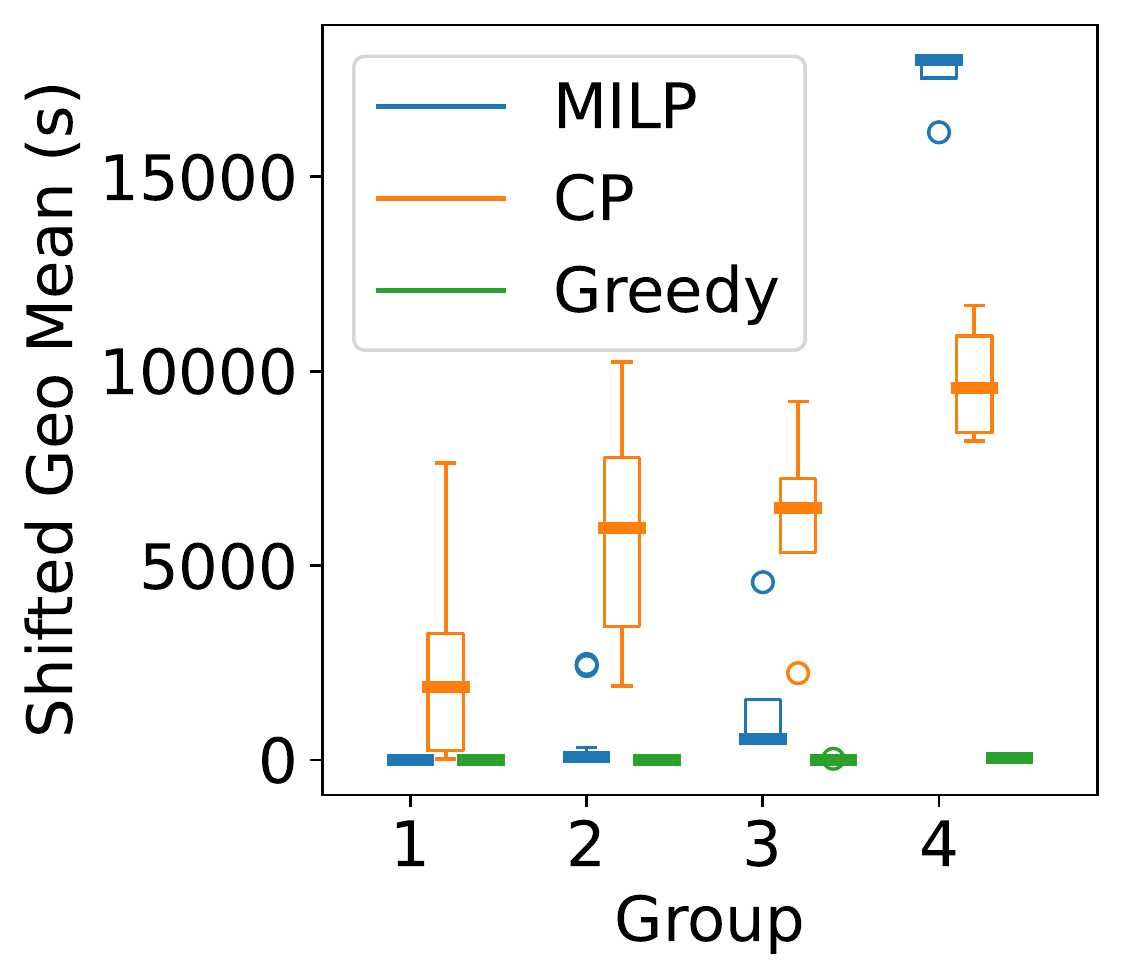}
     \end{subfigure}
     \hfill
     \begin{subfigure}[b]{0.421\columnwidth}
         \includegraphics[width=1\linewidth]{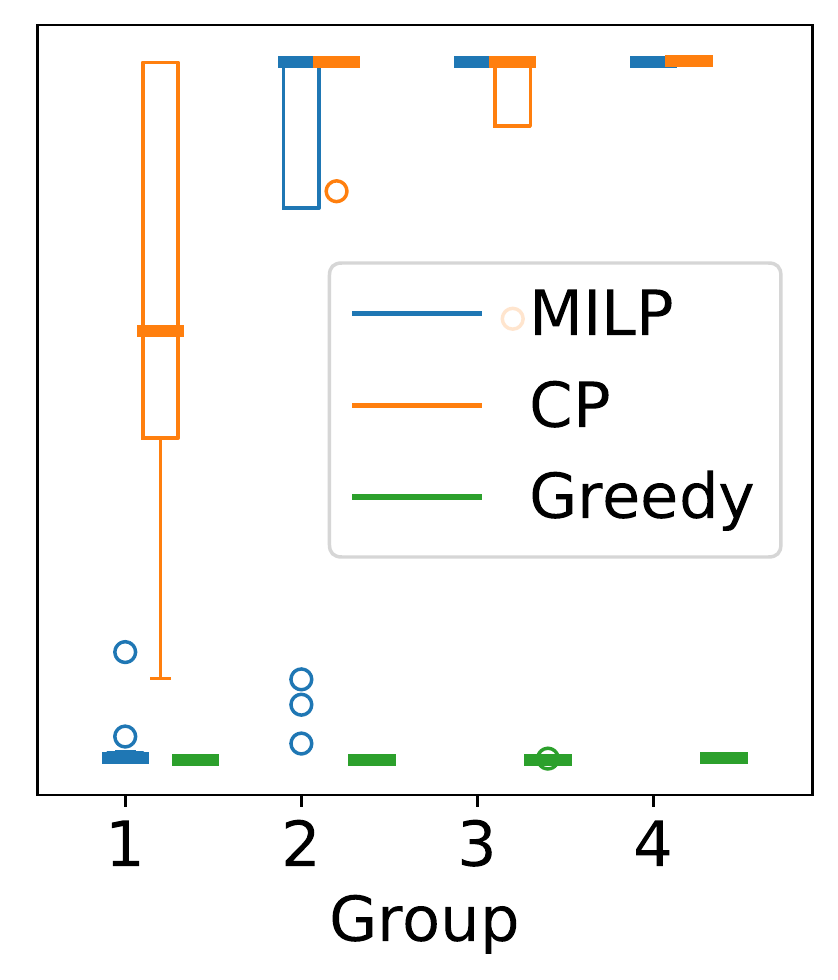}
     \end{subfigure}
  \caption{Shifted Geometric Mean in seconds across $k \in \{1,2,...,9\}$ in SingleChoice (left) \& MultiChoice (right).}
  \label{merged}
\end{figure}
We first study how the computational running times scale with increasing network size ($|M|+|N|$). For each NIA, we measure the Shifted Geometric Mean of the solving time across $k \in \{1,2,...,9\}$ (see~\cite{achterberg2007constraint}, section A.3, for a definition of this widely used summary statistic). Fig. \ref{merged} shows the results for each NIA group. Greedy is orders of magnitude faster than MILP and CP and scales well in both scenarios. The solving times for CP and MILP in the two scenarios show that the difficulty of the problem significantly increases when considering depth of choice. In SingleChoice, the solving time for MILP is shorter than CP in medium and small instances but increases rapidly for large instances, performing worse than CP in Group 4. A possible explanation is that the number of discrete variables of the CP formulation does not depend on $|M|$ or $|N|$ in this scenario, as discussed in Section \ref{sec_models}. In MultiChoice, MILP scales better than CP overall. 



\subsubsection{Solution Quality}
Next, we compare the methods in terms of three metrics: the Mean Relative Error (MRE), the number of instances for which the method found a feasible solution, and the number of instances for which the method proved optimality (Table \ref{sol_quality}). MRE measures the gap between the best solution found by a method and the best solution found across all methods, normalized by the latter. MILP has the lowest MRE for all groups in both scenarios. However, MILP struggles to find feasible solutions for large instances in MultiChoice (Group 4), while CP and Greedy find feasible solutions to all instances. The MRE of Greedy is lower than 0.7\% for all groups in both scenarios and is significantly lower than CP for medium and large instances in MultiChoice. This shows that Greedy produces high-quality solutions as a heuristic even when submodularity does not hold. 

\begin{table}[t]
\resizebox{\columnwidth}{!}{
\begin{tabular}{ll|lrl|lrl}
\toprule
 &  & \multicolumn{3}{|c|}{SingleChoice} &    \multicolumn{3}{|c}{MultiChoice}  \\
\midrule
Group & Method & MRE (\%) &  Feas & Opt &   MRE(\%) &  Feas & Opt \\
\midrule

     &    MILP &     \textbf{0.00} &          99 &        \textbf{99} &  \textbf{0.00} &       99 &     \textbf{97} \\
    1 &     CP &     \textbf{0.00} &          99 &        59 &  0.43 &       99 &     14 \\
(99) & Greedy &     0.29 &          99 &       N/A &  0.66 &       99 &    N/A \\
\midrule
     &    MILP &     \textbf{0.00} &         117 &       \textbf{116} &  \textbf{0.00} &      117 &     \textbf{29} \\
    2 &     CP &     0.01 &         117 &        31 &  3.34 &      117 &      1 \\
   (117)  & Greedy &     0.34 &         117 &       N/A &  0.51 &      117 &    N/A \\
\midrule
     &    MILP &     \textbf{0.00} &          36 &        \textbf{33} &  \textbf{0.00} &       36 &      0 \\
    3 &     CP &     0.04 &          36 &         9 & 12.41 &       36 &      \textbf{1} \\
   (36)  & Greedy &     0.38 &          36 &       N/A &  0.57 &       36 &    N/A \\
\midrule
     &    MILP &     \textbf{0.00} &          27 &         1 &  \textbf{0.05} &       21 &      0 \\
    4 &     CP &     0.91 &          27 &         \textbf{3} & 26.13 &       \textbf{27} &      0 \\
    (27) & Greedy &     0.51 &          27 &       N/A &  0.20 &       \textbf{27} &    N/A \\
\bottomrule
\end{tabular}
}
\caption{Mean relative error (MRE), number of instances for which the method found a feasible solution (Feas) and proved optimality (Opt) for each NIA group. The total number of instances in each group is shown in brackets.}\label{sol_quality}
\end{table}

\subsection{Empirical Evaluation}
What is the effect of solving~\problem~on the WalkScore and travel distances from residents to amenities, in terms of different neighborhoods and residential locations? To answer this question, we use solutions of MILP for instances where it is feasible since MILP has the lowest MRE on average. For instances where MILP did not find a feasible solution, Greedy's solutions are used. 

\subsubsection{Impact on WalkScore}
Fig.~\ref{fig_change_nia} shows the change in the average WalkScore for each NIA when additional amenities are introduced by optimization. Most NIAs lie in the mostly-industrial northwest and the suburban/rural northeast of Toronto, where infrastructure is limited. In contrast, few NIAs lie in the urban core in the south where amenities are dense, which is consistent with the reported walkability of Toronto \cite{report,report2}. The change in WalkScore varies across NIAs. Adding 3 amenities of each type improves the WalkScore by more than 50 for 4 NIAs. We observe that NIAs with low current WalkScore show greater improvement after allocation; current WalkScore for each NIA in Appendix \ref{appendix_cur_walk}.

\begin{figure}[t]
\centering
 \includegraphics[width=1\columnwidth]{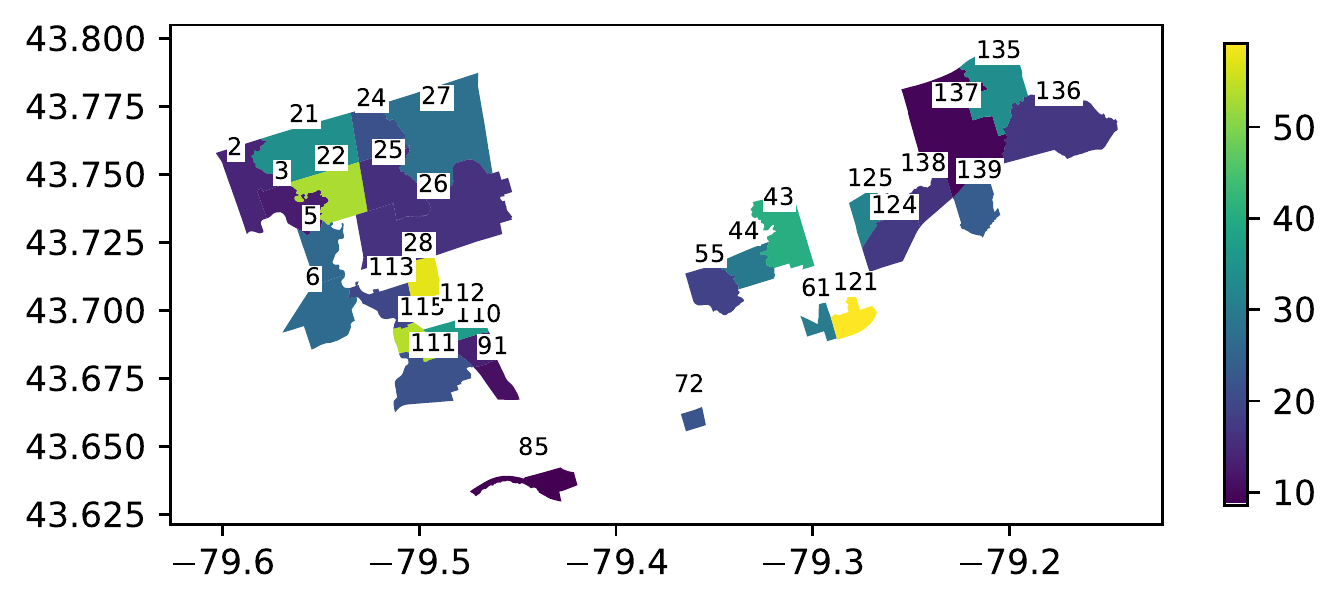}
  \caption{Change in average WalkScore with MultiChoice~\problem~($k=3$). NIAs are labeled by their IDs (neighborhood names in Appendix \ref{appendix_nia_name}).}\label{fig_change_nia}
\end{figure}


\begin{figure}[t]
  \begin{subfigure}[b]{0.49\columnwidth}
         \includegraphics[width=1\linewidth]{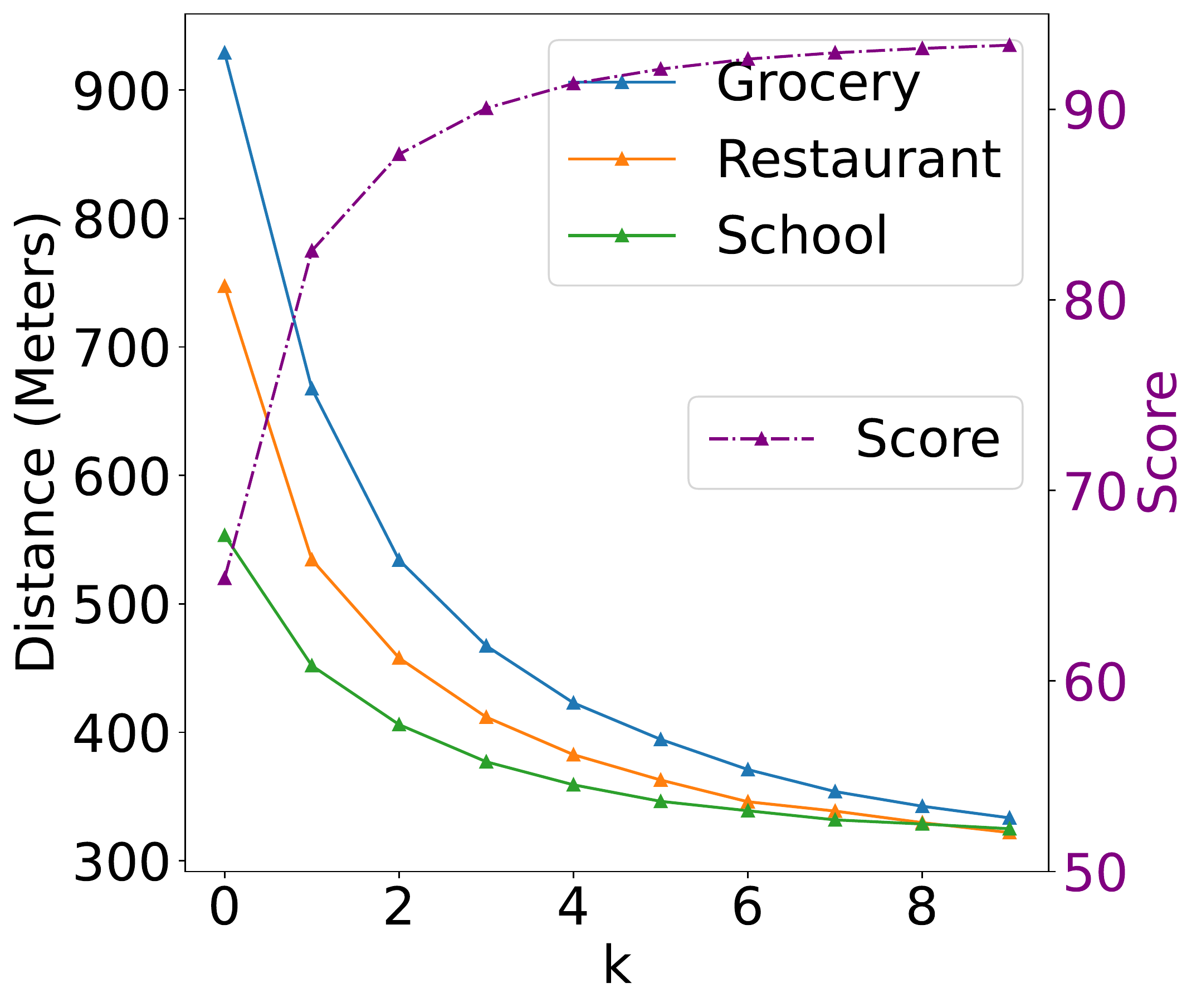}
     \end{subfigure}
     \hfill
     \begin{subfigure}[b]{0.49\columnwidth}
         \includegraphics[width=1\linewidth]{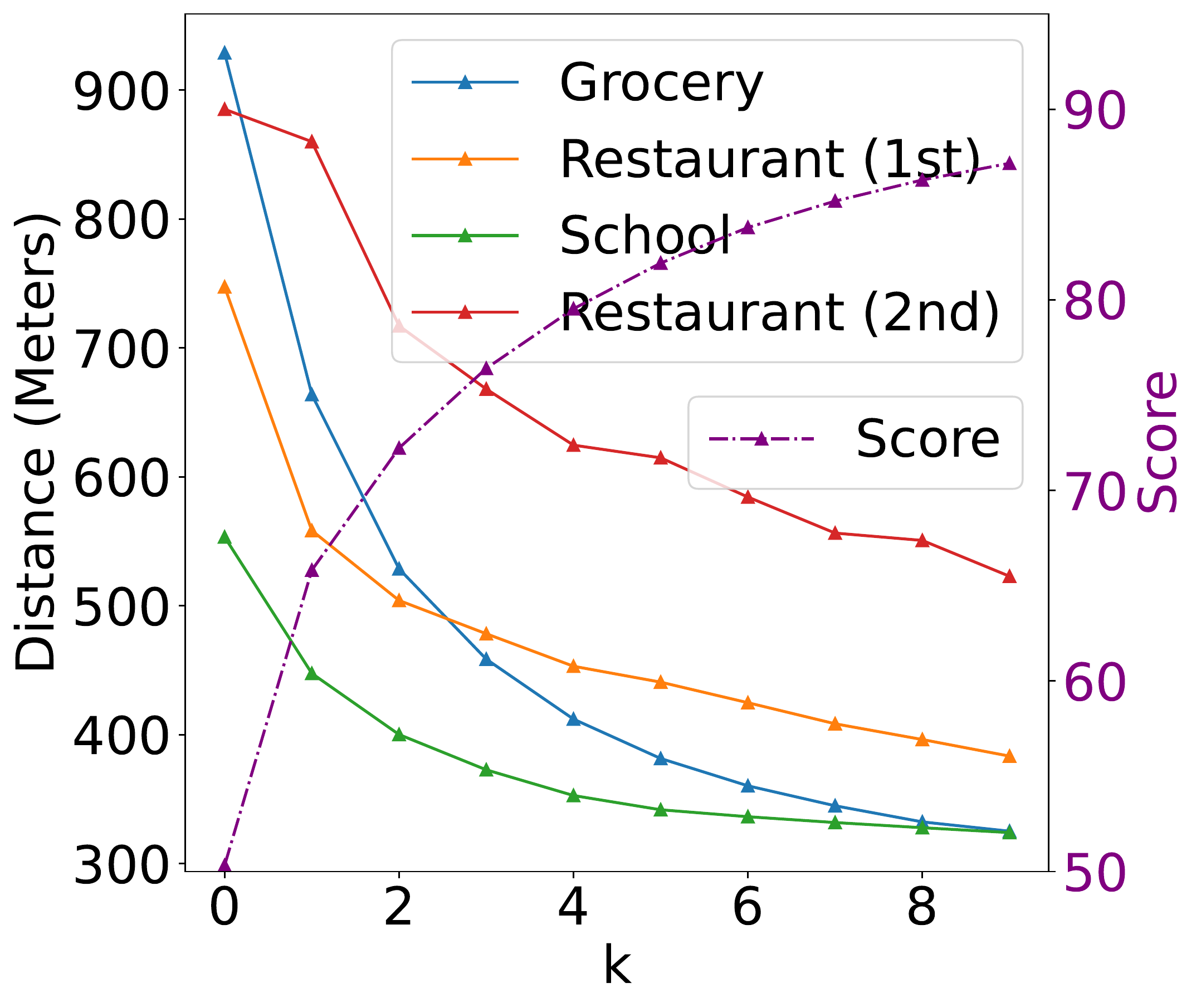}
     \end{subfigure}
  \caption{Average WalkScore and travel distances to different amenity types w.r.t. $k$ in SingleChoice (left) \& MultiChoice (right). For MultiChoice, distances to the 1st and 2nd restaurants are shown.}
  \label{obj_vs_k}

\end{figure}

Additionally, we show how the average WalkScore across all NIAs changes w.r.t. the value of $k$ (Fig. \ref{obj_vs_k}). The objective exhibits diminishing returns as $k$ increases in the SingleChoice scenario but not in MultiChoice; this agrees with the submodularity analyses of Section \ref{sec_theory}. \todo{In the MultiChoice case, setting $k=3$ increases the average objective by more than 25. } Looking at the travel distances, we see that the framework effectively reduces walking distances for all types/choices considered.

\begin{figure*}[ht]
     \centering
     \begin{subfigure}[b]{0.245\textwidth}
         \centering
         \includegraphics[width=\textwidth]{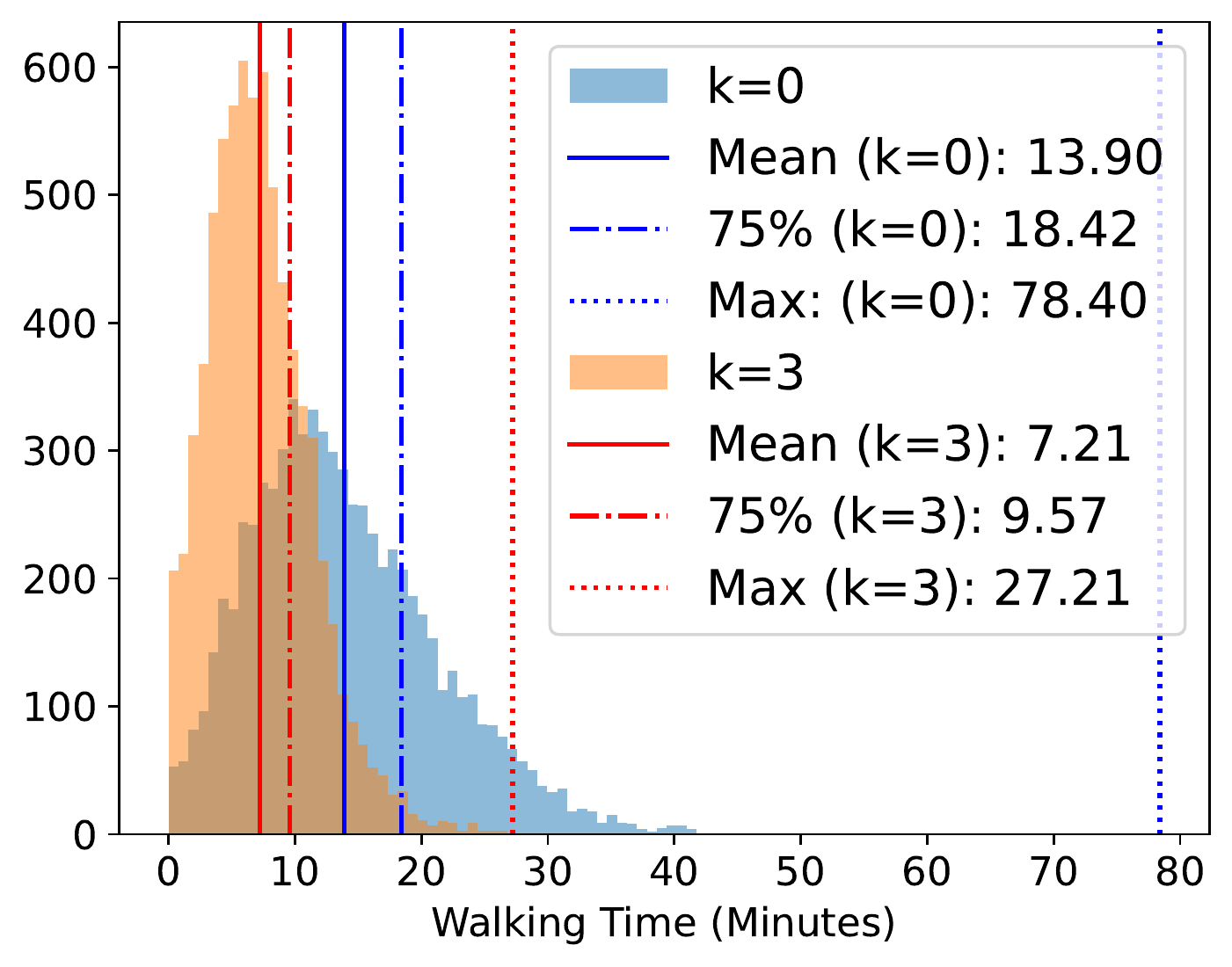}
         \caption{Grocery}
         \label{fig:y equals x}
     \end{subfigure}
     \hfill
     \begin{subfigure}[b]{0.245\textwidth}
         \centering
         \includegraphics[width=\textwidth]{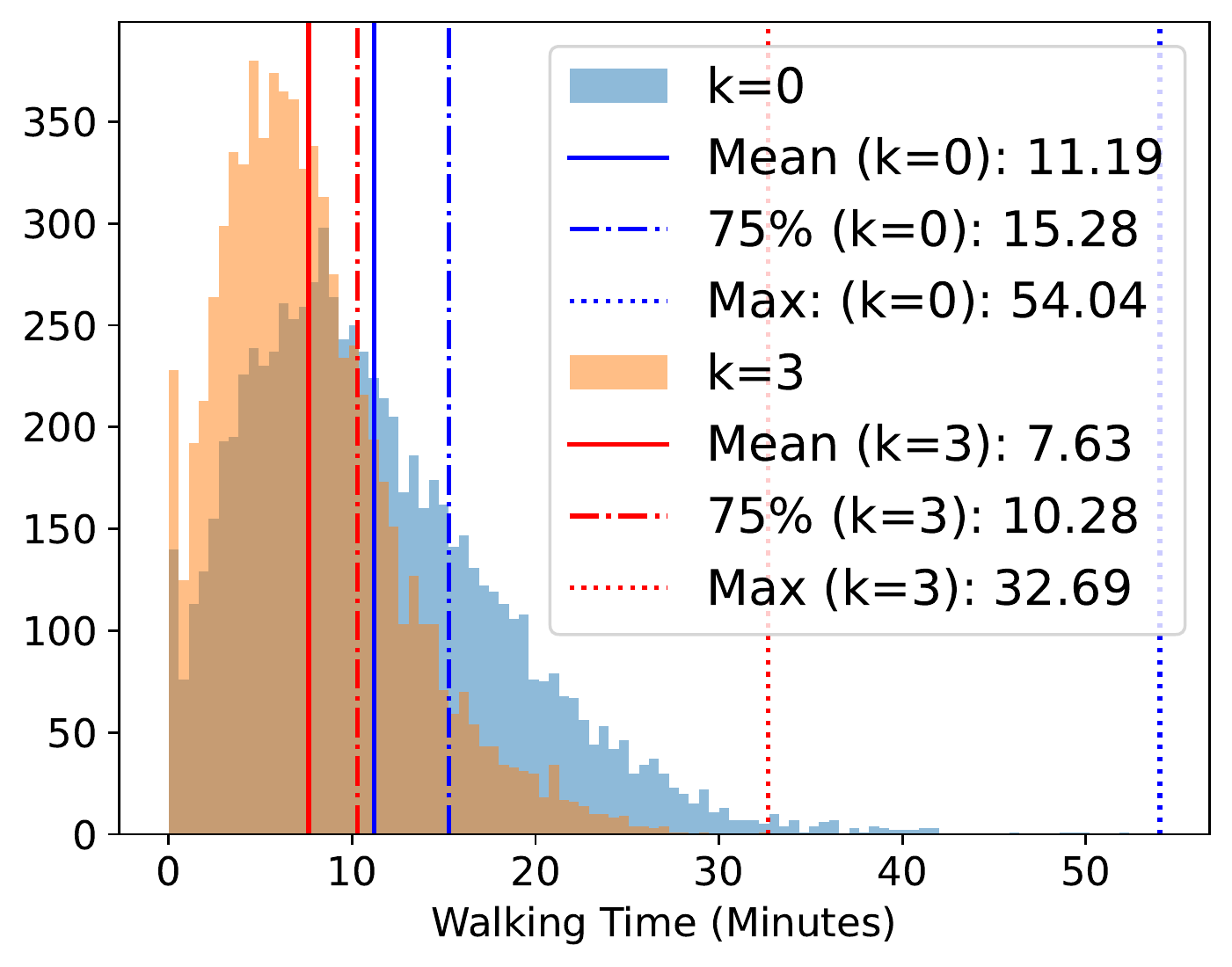}
         \caption{1st Nearest Restaurant}
         \label{fig:three sin x}
     \end{subfigure}
     \hfill
     \begin{subfigure}[b]{0.245\textwidth}
         \centering
         \includegraphics[width=\textwidth]{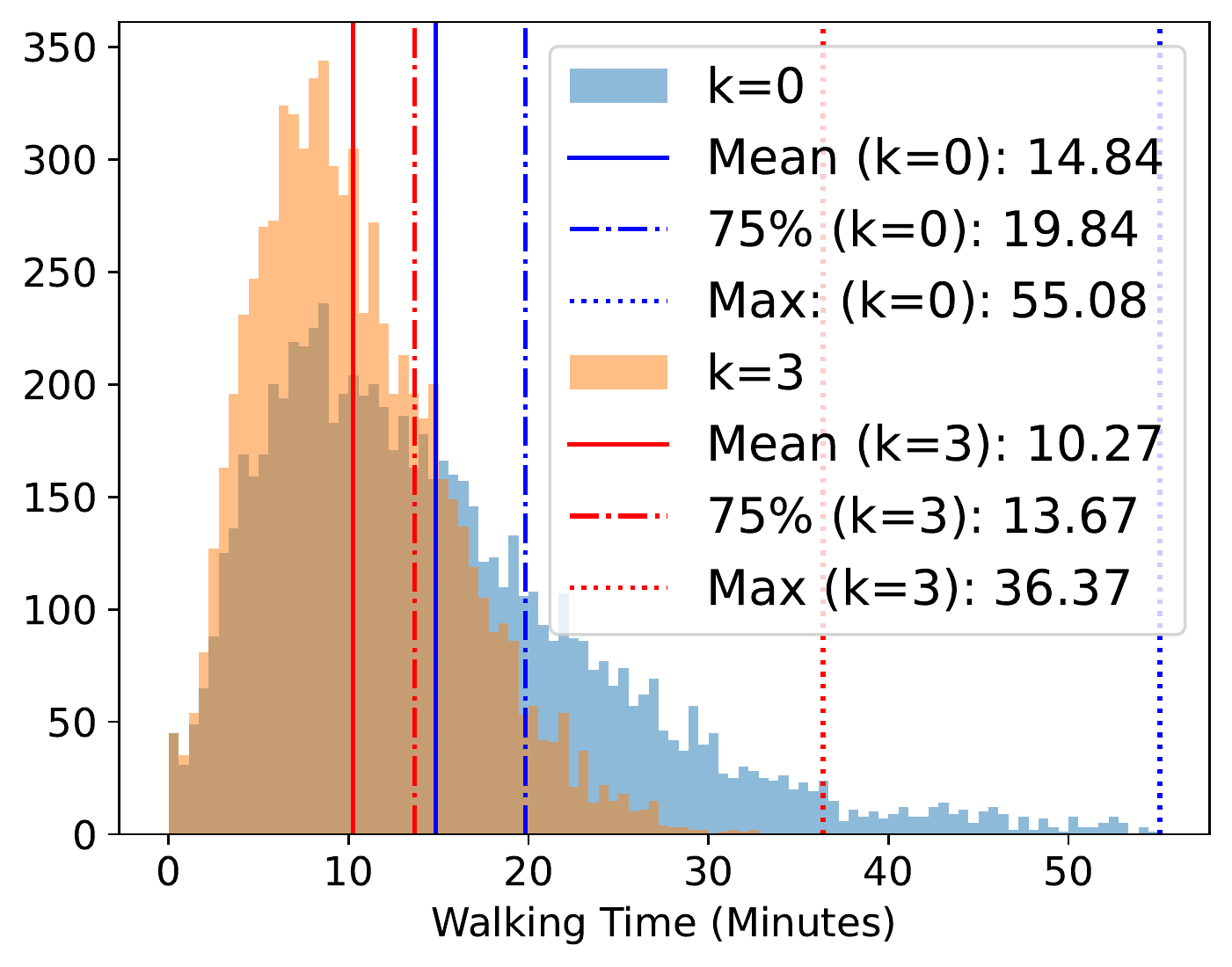}
         \caption{2nd Nearest Restaurant}
         \label{fig:five over x}
     \end{subfigure}
        \hfill
     \begin{subfigure}[b]{0.245\textwidth}
         \centering
         \includegraphics[width=\textwidth]{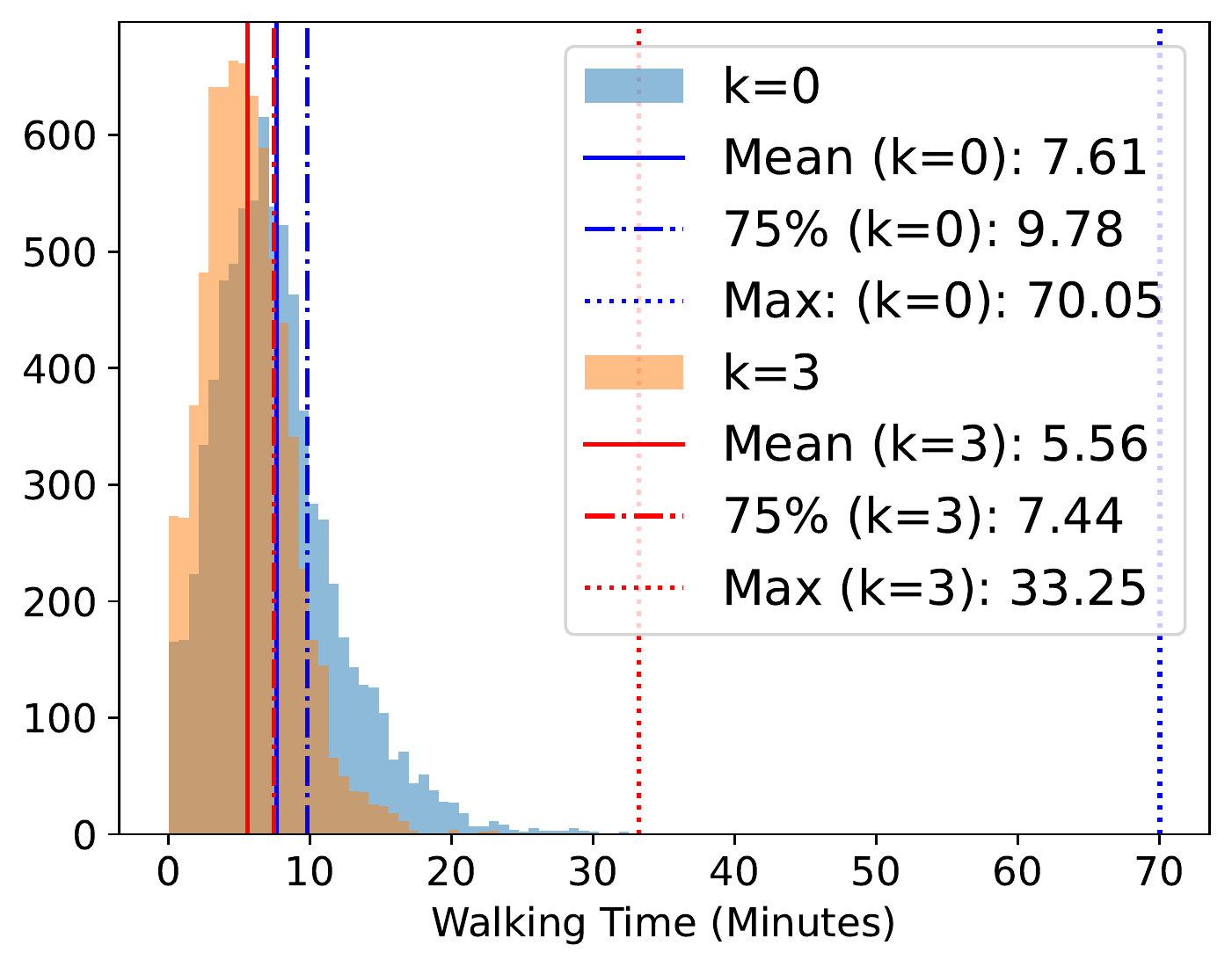}
         \caption{School}
         \label{fig:five over x}
     \end{subfigure}
     \caption{Histogram of walking times to different amenity types across all residential nodes in all neighbourhoods. Adding $k=3$ amenities shifts the histogram of walking times to the left (in orange), which also translates into smaller mean/maximum/75th percentile walking times relative to not adding any amenities (in blue).}
        \label{fig:hist}
    
\end{figure*}

\subsubsection{Individual Residential Locations}
We analyze the impact of~\problem~on individual residential locations using a histogram of walking times to the 3 amenity types for all residential nodes in all 31 NIAs (Fig. \ref{fig:hist}). Distance-to-time conversion is done using a walking speed of $1.2m/s$~\cite{walk_time}. In the MultiChoice case, an allocation with $k=3$ reduces the walking distances of the 75th percentile of all individual residential locations to 10 minutes for all amenity types. According to \citet{report}, a residential preference survey reveals that a 10-minute walking distance to stores and services characterizes a walkable neighbourhood. For grocery stores, this reduction in distance, of up to half relative to the current state, is significant. For schools, we do not observe a large improvement in the mean or the 75th percentile; most residents can walk to a school within 10 minutes currently. As schools are non-commercial, their locations may have been well-optimized historically. However, we do observe a large reduction in the \textit{maximum} walking times to schools.


\begin{figure}[t]

\centering

\includegraphics[width=\columnwidth]
{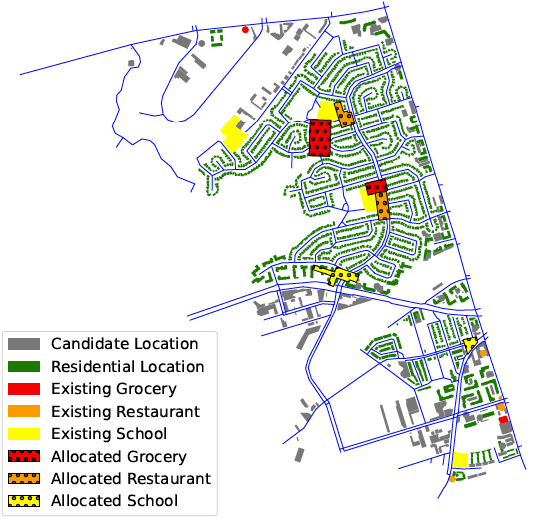}
  \caption{Allocated amenities at Victoria Village: 2 newly introduced groceries, restaurants, and schools.}\label{viz}

\end{figure}

\subsubsection{Visualization of allocated amenities}
Fig.~\ref{viz} illustrates a~\problem~solution for NIA Victoria Village along with existing residential locations, candidate allocation locations, and existing amenities. Allocated amenities seem to fall at the heart of residential clusters. Perhaps more interestingly, some newly allocated amenities are very close to the locations of other types of amenities (existing or allocated) and seem to form an urban center with a mix of different types.

\section{Related Work}
\subsection{Walkability Optimization}
The Introduction already discusses some of the most relevant work that uses genetic algorithms. To further elaborate, \citet{gen1} optimizes for each different amenity type independently, which may hinder optimality. Moreover, existing works suffer from unrealistic assumptions that limit the applicability of the framework and the quality of empirical evaluation such as designing street grid patterns from scratch \cite{Lima2022AGO} and allocating amenities to empty street layout~\cite{gen2}. \citet{gen3} performs a case study in an existing city but only uses randomly sampled residential units and ignores current amenity locations, generating solutions that override existing infrastructure. 


\subsection{$k$-Median and Facility Location Problem (FLP)}
Compared to the FLP, \problem~considers multiple facility types and depth of choice. Without these, \problem~is equivalent to the Submodular FLP defined in \cite{submodular_flp_2}, and previous work has shown that objective~\eqref{walkscore_obj} is submodular in this case \cite{submodular_flp_1}. 

Moreover, the objective function WalkScore does not satisfy the properties of a metric space, in contrast to the closely related $k$-median problem. Algorithms for $k$-median that provide better approximation ratios than standard greedy include reverse greedy \cite{Chrobak2005TheRG}, local search \cite{Arya2001LocalSH}, LP relaxation \cite{lp1,Charikar2012ADL}, and Lagrangian relaxation \cite{10.1145/375827.375845,Jain2003GreedyFL}. However, these algorithms assume that the objective function is defined on a metric space and are not applicable to our problem.


\section{Conclusion and Discussion}
Automobile-reliant communities with limited access to amenities in their vicinity have a great potential for transformation into more walkable and sustainable neigbhbourhoods. We formulate the problem of Walkability Optimization where amenities are introduced at strategic locations to improve the proximity to residents. Our \problem~formulation realistically models residents' behaviour by integrating multiple amenity types, depth of choice, and an objective function representing the proximity to amenities. We also take into account existing amenities in the context of built cities. We provide MILP and CP formulations and an efficient greedy algorithm motivated by the submodular structure of the~\problem~objective (without depth of choice). An experimental evaluation on high-quality data from Toronto shows that MILP and Greedy are effective at producing near-optimal solutions, with a scalability advantage for the latter. Our framework produces solutions that significantly improve the walkability in underserved neighbourhoods on average and reduce the walking distances for individual residential locations.

While we have prioritized incorporating realistic facets of walkability optimization into our formulation, more can potentially be done by: considering the population at each residential location, the construction cost at each candidate allocation location, the area/size of candidate locations, and the service capacity of amenities. If the data is available, these can be easily integrated into the objective and constraints. In addition, \problem~quantifies walkability in terms of travel distances without considering other factors that may affect accessibility such as the safety/quality of walking paths, which can potentially be incorporated into the formulation by applying penalties appropriately. Lastly, our experiments were based on neighbourhood-scale instances; testing our methods at full city-scale might be of future interest.


\clearpage
\bibliography{aaai23}

\appendix

\clearpage
\section{Proof of NP-Completeness}\label{appendix_np_compleness}

We prove that this problem is not solvable in polynomial time assuming that P$\neq$ NP. Consider a simplified version of our problem that considers only one type of amenity $a\in A^{plain}$ with capacity $c_j=1$ for each candidate allocation location $j \in M$. Additionally, we replace the WalkScore $f()$ with an affine objective function $f(x)=-rx + b$, where $r>0,b>0$. The decision version of our problem is: given a set $N$ of residential locations, a set $M$ of candidate locations, and a set of existing locations $L$,  does there exist a set $S \subseteq M$ with $|S|\leq k$ such that $\sum_{i \in N}f(\min_{j \in S\cup L} d_{ij}) \geq t_1$, for arbitrary $t_1$?

First, the problem is in NP.
\begin{itemize}
    \item Input: Set of residential locations $N$, set of potential allocation locations $M$, set of existing locations $L$, pairwise distances $d_{ij}$ for $i \in N,j \in M \cup L$, a positive integer $k$, a threshold $t_1$, and the function $f()$
    \item Evidence: a subset of nodes $S \in M$
    \item Requirement: $\sum_{i \in N}f(D_i) \geq t_1$ and $|S| \leq k$, where $D_i = \min_{j \in S \cup L}d_{ij}$
    \item Algorithm for checking: For each $i \in N$, compute $f_i^*=\operatorname*{argmax}_{j \in S \cup L} f(d_{ij})$. Then, compute $\sum_{i \in N}f_i^*$.
\end{itemize}
Then, we show that the problem is NP-hard by reducing the $k$-median problem to our problem. The decision version of the $k$-median problem is the following: Given a metric space $(F \cup C,\rho)$  where $F$ is a set of facility locations, $C$ is a set of clients and $\rho()$ is a distance function in the metric space, does there exist a $k$-element set $S \in F$ of open facilities such that $\sum_{i \in C}\min_{j \in S}\rho(i,j)  \leq t_2$? 

Given an instance of the $k$-median problem with $F$, $C$, distance function $\rho()$, threshold $t$, and cardinality constraint $k$, we construct an instance of our problem with a set of residential nodes $N=C$ and a set of potential allocation nodes $M=F$ with distances $d_{ij}=\rho(i,j), \; \forall i \in N,j \in M$. We add a single existing amenity $L=\{q\}$ such that $d_{iq}>\max_{j \in M}d_{ij},\forall i \in N$. This reduction can be performed in polynomial time, as desired.

We show that $S$ with $|S|=k$ satisfies $\sum_{i \in C}\min_{j \in S}\rho(i,j)  \leq t$ in the $k$-median problem if and only if $S$ satisfies $\sum_{i \in N}f(\min_{j \in S\cup L} d_{ij}) \geq (-tr + b|C|)$ in our problem.

Since we constructed the instance such that the distances to existing amenity $L=\{q\}$ are greater than any candidate allocation location, we have 
\begin{equation*}
    \min_{j \in S\cup L} d_{ij}=\min_{j \in S} d_{ij}, \forall i \in N.
\end{equation*}
And because of the way we define the distances in the instance of our problem, we have
\begin{equation*}
    \min_{j \in S} d_{ij}=\min_{j \in S}\rho(i,j), \forall i \in N.
\end{equation*}
Then the following are equivalent:

\begin{equation*}
\begin{split}
    &\sum_{i \in C}\min_{j \in S}\rho(i,j) \leq t \\
    &\leftrightarrow \sum_{i \in C}(-r\min_{j \in S}\rho(i,j)+b) \geq -tr + b|C|\\ 
    &\leftrightarrow \sum_{i \in N}(-r\min_{j \in S} d_{ij}+b) \geq -tr + b|C|\\ 
    &\leftrightarrow \sum_{i \in N}f(\min_{j \in S} d_{ij}) \geq -tr + b|C|\\ 
    &\leftrightarrow \sum_{i \in N}f(\min_{j \in S \cup L} d_{ij}) \geq -tr + b|C|.
\end{split}
\end{equation*}

In words, a feasible solution to the $k$-median instance is feasible in the corresponding instance of our problem, and vice versa. Combined with the fact that our problem is in NP and a polynomial-time reduction is possible, we conclude that the decision version of our problem is NP-Complete. 




\section{No Submodularity with Depth of Choice: a Counter Example}\label{appendix_counter}
For the case with depth of choice, we provide a counter-example to show that submodularity does not hold.

We consider the weighted distances to three amenity types: grocery stores, restaurants, and schools. For restaurants, we consider the distances to the 10 nearest instances. The raw weights for different amenity types and different options are from Appendix \ref{appendix_amenity_weights}, and the normalized values of the weights are shown in Table \ref{table_counter_example}. Consider a network that has only one residential location $N=\{n\}$. The existing amenities in the network include 6 restaurants, a grocery store, and a school, and all the existing amenity locations are $2000$m away from the residential location. The set of potential locations in the networks is  $M=\{1,2,3,4,5,6,7\}$, $d_{nj}=1800$m for $j \in \{1,2,3,4,5,6\}$, and $d_{n7}=1$m.

Consider the two solution sets. Solution set $S$ has allocated additional restaurants to potential locations $\{1,2,3,4\}$. Solution set $T$ has allocated additional restaurants to potential locations $\{1,2,3,4,5,6\}$. Clearly, $S \subseteq T$. Element $e$ places a restaurant to location $7$, so we have $e \in V \setminus T$. The distances to the nearest instance of grocery stores and schools and the distances to the top 10 nearest instances of restaurants are shown in Table \ref{table_counter_example}.

\begin{table}[h]
\resizebox{\columnwidth}{!}{
\begin{tabular}{llllll}
\toprule
Amenity     & Weight & Dist. $S$    & Dist. $T$    & Dist. $S \cup \{e\}$ & Dist. $T \cup \{e\}$\\
\midrule
Grocery & 0.43  & 2000 & 2000 & 2000     & 2000     \\
Restaurant, 1   & 0.11  & 1800 & 1800 & 1        & 1        \\
Restaurant, 2       & 0.06  & 1800 & 1800 & 1800     & 1800     \\
Restaurant, 3       & 0.04  & 1800 & 1800 & 1800     & 1800     \\
Restaurant, 4       & 0.04  & 1800 & 1800 & 1800     & 1800     \\
Restaurant, 5       & 0.03  & 2000 & 1800 & 1800     & 1800     \\
Restaurant, 6       & 0.03  & 2000 & 1800 & 2000     & 1800     \\
Restaurant, 7       & 0.03  & 2000 & 2000 & 2000     & 1800     \\
Restaurant, 8       & 0.03  & 2000 & 2000 & 2000     & 2000     \\
Restaurant, 9       & 0.03  & 2000 & 2000 & 2000     & 2000     \\
Restaurant, 10      & 0.03  & 2000 & 2000 & 2000     & 2000     \\
School  & 0.14  & 2000 & 2000 & 2000     & 2000  \\
\bottomrule
\end{tabular}
}
\caption{Distances to different amenity types (including different options for restaurants) under different solution sets.}
\label{table_counter_example}
\end{table}
According to the values in Table \ref{table_counter_example}, we calculate the weighted distances $l_n$ and WalkScore values $f(l_n)$ under different solution sets, which is shown in Table \ref{table_counter_example_obj}.

\begin{table}[h]
\resizebox{\columnwidth}{!}{
\begin{tabular}{llllll}
\toprule
Solution Set   & $S$  &  $T$    & $S \cup \{e\}$ & $T \cup \{e\}$\\
\midrule
Weighted Distance ($l_n$) &  1950.0 & 1938.0 & 1746.11     & 1734.11     \\
Walkability Score ($f(l_n)$)     & 7.5 & 7.7 & 13.27        & 14.00        \\
\bottomrule
\end{tabular}
}
\caption{Objective under different solution sets.}
\label{table_counter_example_obj}
\end{table}

Since this network only has one residential location, the objective function is $F=f(l_n)$. Clearly, we have $\Delta_F{(e|S)}:=F(S \cup \{e\}) - F(S)=5.77$ and $\Delta_F{(e|T)}:=F(T \cup \{e\}) - F(T) =6.30$, so $\Delta_F{(e|S)} < \Delta_F{(e|T)}$. This violates the submodularity structure.

\section{Constraint Programming Model  (CP)}\label{appendix_cp}
\subsubsection{Variables}
Our constraint programming model uses the
index of the nodes to describe the selected locations for the amenities to be allocated and the assignment of amenities to residential locations. First, for allocation, variable $y_{k'a}$ indicates the index of the location of the $k'^{th}$ allocated instance of amenity type $a$. Second, for the assignment of amenities $a \in A^{depth}$ to residential locations, variables $x_{ia}^p$ indicate the index of location that residents at location $i$ visit for the $p^{th}$ nearest instance of type $a, a \in A^{depth}$. However, for $a \in A^{plain}$, assignment variables are not needed.

We define the domain for these two sets of variables. Note that in the CP model, elements in sets $M$, $N$, and $L_a, a \in A$ are indices of nodes in the network. Sets $M$, $N$, and $L_a, a \in A$ are disjoint. Then, the domains of $y_{k'a}$ and $x_{ia}^p$ are:
\begin{equation*}
    y_{k'a} \in M \cup \{dummy\},\forall k' \in \{1,...,k_a\},\forall a \in A,
\end{equation*}
\begin{equation*}
    x_{ia}^p \in M \cup L_a, \forall i \in N,\forall p \in P_a^{Y},\forall a \in A^{depth}.
\end{equation*}
A dummy node with unlimited capacity is introduced to account for the case when there are not enough locations in $M$ to allocate all amenity instances. Residents will not be assigned to the dummy node. 

In addition, we have continuous variables that describe the distances and WalkScore. For $a \in A^{plain}$,  $z_{iak'}$ describes the distance from location $i$ to the $k'^{th}$ allocated instance of amenity type $a$. For $a \in A^{depth}$, $z_{ia}^p$ describes the distance from $i$ to the $p^{th}$ nearest instance of $a$. Finally, $l_{i}$ and $f_{i}$ denote the weighted distance and the WalkScore at residential location $i$, respectively. 

\subsubsection{Constraints}
First, we enforce the requirements on the capacity of the candidate allocation locations. Let $Y$ be an array consisting of all $y_{k'a}$ variables such that $ k' \in \{1,...,k_a\},a\in A$. We ensure that the number of allocated instances across all amenity types at each candidate location $j$ does not exceed the node capacity.
\begin{equation*}
    \text{count}(Y,j) \leq c_j , \forall j \in M
\end{equation*}
Next, we calculate the weighted walking distances $l_i$ according to Eqn. (\ref{walk_formula}). For $a \in A^{plain}$, the distance to the nearest instance can be explicitly expressed as the minimum across all existing and allocated instances. For $a \in A^{depth}$, we take the weighted combination of the distances to the top-$r$ nearest instances:
\begin{equation*}\label{dist_cp}
\begin{split}
l_i & = \sum_{a \in A^{plain}} w_a \min(\min_{k' \in \{1,...,k_a\}} z_{iak'},\min_{j \in L_a} d_{ij})   \\
    & + \sum_{a \in A^{depth}} (\sum_{p \in P_a^{Y}} w_{a}^p z_{ia}^p + \sum_{p \in P_a^{N}} w_{a}^p D^{\infty}), \forall i \in N.
\end{split}
\end{equation*}
Variables $z_{iak'}$ and $z_{ia}^p$ in the constraint above are described using element constraints. For each $i \in N$, let $arr^i$ be an array where the $t^{th}$ element of $arr^i$ is:
\begin{equation*}
    arr^i[t] = d_{it}, \forall t \in \cup_{a \in A}L_a \cup M, \forall i \in N.
 \end{equation*}
Then, the distances from $i$ to amenity locations can be obtained by array indexing. Since $y_{k'a}$ is the node index of $k'^{th}$ instance of $a$, we have
\begin{equation*}
    z_{iak'}=arr^i[y_{k'a}], \forall i \in N, \forall a \in A^{plain},\forall k' \in \{1,...,k_a\}.
\end{equation*}
Similarly, $z_{ia}^p$ can be expressed as:
\begin{equation*}
    z_{ia}^p=arr^i[x_{ia}^p], \forall i \in N, \forall a \in A^{depth},\forall p \in P_a^Y.
\end{equation*}

Since assignment variables are used for $a \in A^{depth}$, we also ensure that any location should be allocated before they are assigned:
\begin{equation*}
\begin{split}
     \text{any}_{i\in N,p \in P_a^Y}(x_{ia}^p=j) \Rightarrow \text{any}_{k' \in \{1,...,k_a\}}(y_{k'a}=j),\\
     \forall j \in M, \forall a \in A^{depth}.
\end{split}
\end{equation*}

As mentioned in the MILP model, for $a \in A^{depth}$, each choice $p \in P_a^{Y}$ should be a different instance of $a$. When the choice corresponds to an existing amenity ($j \in L_a$), we ensure that the instance appears only once among all choices for $i$: 
\begin{equation*}
    \text{count}([x_{ia}^1,...,x_{ia}^{|P_a^Y|}],j) \leq 1, \forall i \in N, \forall j \in L_a,\forall a \in A^{depth}.
\end{equation*}
When the choice corresponds to candidate locations ($j \in M$), we ensure that the number of choices provided at $j$ does not exceed the number of instances allocated to $j$:
\begin{equation*}
\begin{split}
    \text{count}([x_{ia}^1,...,x_{ia}^{|P_a^Y|}],j) \leq  \text{count}([y_{1a},...,y_{k_aa}],j), \\ \forall i \in N, \forall j \in M,\forall a \in A^{depth}.
\end{split}
\end{equation*}

Since index-based models can introduce symmetry in allocation, we break the symmetry by specifying the order of node indices between instances of the same type:
\begin{equation*}
    y_{k_1a} \leq y_{k_2a}, \forall a \in A, \forall k_1,k_2 \in \{1,...,k_a\}, k_1<k_2.
\end{equation*}


Lastly, we ensure that the PWL relationship in Eqn. (\ref{pwl_of_dist}) holds.

\subsubsection{Objective} The objective is to maximize $F$ in Eqn. (\ref{walkscore_obj}).

\subsubsection{Note}
Note that in the CP model, all continuous variables are entirely driven by discrete variables $y_{k'a}$ and $x_{ia}^p$ in the search. It was found that this approach is more efficient than discretizing the continuous variables, as there is no need for the solver to search over the continuous variables given that they are explicitly linked to discrete decision variables. Discretizing the distances and scores can create a very large search space. Also, our CP model uses an index-based formulation, which was found to be more efficient than a binary formulation. 

\section{Amenity Weights}\label{appendix_amenity_weights}
The raw weights for grocery stores, restaurants, and schools are obtained from the WalkScore methodology \cite{walk_score} and the values are as follows. For restaurants, the WalkScore methodology considers the weights for the top 10 nearest options. In our case study, these raw weights are normalized so that the weights for all amenities and all options sum up to 1.

\begin{quote}
Grocery: [3],\\
Restaurants: [.75, .45, .25, .25, .225, .225, .225, .225, .2, .2],\\
Schools: [1],
\end{quote}

\section{Piecewise-Linear Approximation for WalkScore}\label{appendix_pwl}
The piecewise-linear WalkScore approximation is parameterized by $\bar{t}$, the set of breakpoints in the piecewise-linear function (Table~\ref{pwl_param}).
\begin{table}[h]
\begin{tabular}{llllll}
Distance (meters)   & 0 &  400    & 1800 & 2400\\
WalkScore   &  100 & 95 & 10    &0  \\
\end{tabular}
\caption{WalkScore Parameters.}
\label{pwl_param}
\end{table}

\section{Neighbourhood Improvement Areas}\label{appendix_nia_name}
The neighbourhood IDs and names of the 31 Neighbourhood Improvement Areas (NIAs) studied are shown in Table \ref{nia_id_mapping}. This data is publicly available from The City of Toronto’s Open Data Portal \cite{data_NIA}.
\begin{table}[h]
    \centering
\begin{tabular}{rl}
\toprule
 ID &                             Name \\
\midrule
         2 & Mount Olive-Silverstone-Jamestown \\
         3 &      Thistletown-Beaumond Heights \\
         5 &                  Elms-Old Rexdale \\
         6 &     Kingsview Village-The Westway \\
        21 &                    Humber Summit \\
        22 &                       Humbermede \\
        24 &                      Black Creek \\
        25 &           Glenfield-Jane Heights \\
        26 &             Downsview-Roding-CFB \\
        27 &          York University Heights \\
        28 &                           Rustic \\
        43 &                 Victoria Village \\
        44 &                  Flemingdon Park \\
        55 &                 Thorncliffe Park \\
        61 &                    Taylor-Massey \\
        72 &                      Regent Park \\
        85 &                   South Parkdale \\
        91 &               Weston-Pellam Park \\
       110 &        Keelesdale-Eglinton West \\
       111 &               Rockcliffe-Smythe \\
       112 &         Beechborough-Greenbrook \\
       113 &                          Weston \\
       115 &                    Mount Dennis \\
       121 &                        Oakridge \\
       124 &                    Kennedy Park \\
       125 &                         Ionview \\
       135 &                     Morningside \\
       136 &                       West Hill \\
       137 &                          Woburn \\
       138 &                   Eglinton East \\
       139 &             Scarborough Village \\
\bottomrule
\end{tabular}
    \caption{IDs and neighbourhood names of the 31 Neighbourhood Improvement Areas (NIAs) in the City of Toronto.}
    \label{nia_id_mapping}
\end{table}

\section{Current WalkScore}\label{appendix_cur_walk}
Current WalkScores for each NIA are shown in Fig. \ref{cur_nia_score}.

\begin{figure}[h]
\centering
 \includegraphics[width=1\columnwidth]{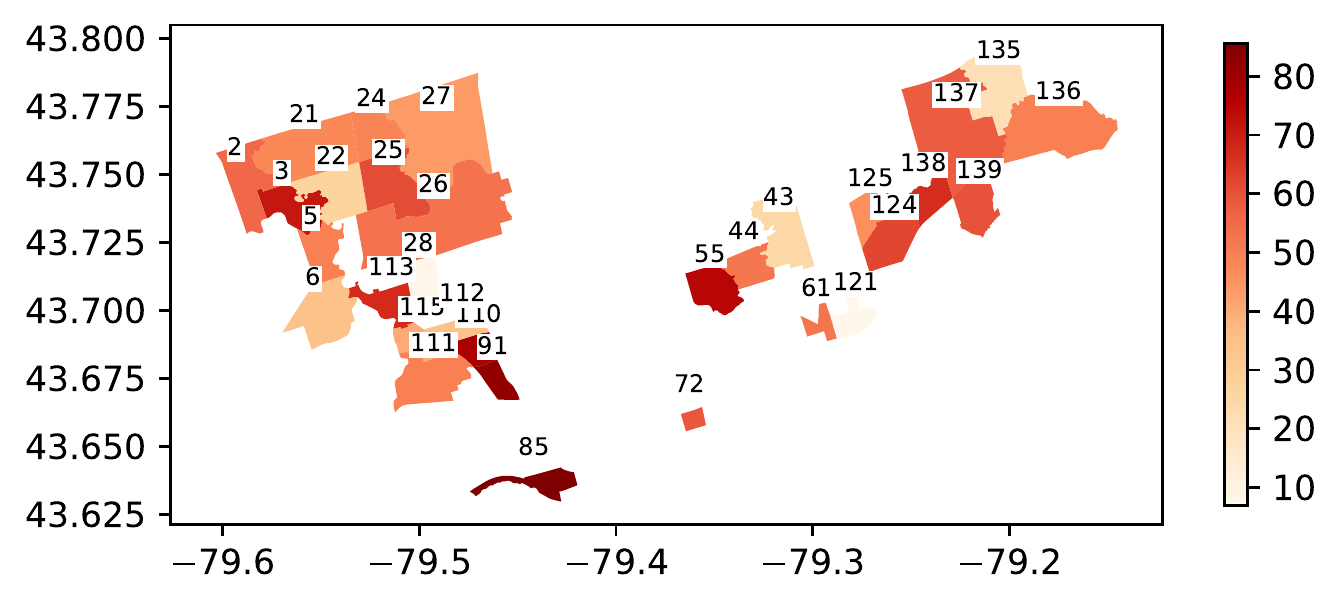}
  \caption{Current average WalkScore with MultiChoice~\problem.}\label{cur_nia_score}
\end{figure}

\end{document}